\documentclass[12pt]{article}
\usepackage[utf8]{inputenc}
\usepackage{amsmath,amssymb,amsthm,bbm,srcltx,graphicx,a4wide,xr,xargs, mathtools}
\usepackage[a4paper,top=2.5cm,bottom=2.5cm,left=2.5cm,right=2.5cm]{geometry}
\allowdisplaybreaks
\usepackage{subfig}
\usepackage{graphics}
\usepackage[]{xcolor}
\usepackage[shortlabels]{enumitem}
\usepackage{indentfirst}
\usepackage{hyperref}
\usepackage{breakurl}
\usepackage[T1]{fontenc}

\newtheorem{theorem}{Theorem}[section]
\newtheorem{lemma}[theorem]{Lemma}
\newtheorem{proposition}[theorem]{Proposition}
\newtheorem{corollary}[theorem]{Corollary}

\newtheorem{assumption}[theorem]{Assumption}

\theoremstyle{remark}
\newtheorem{remark}[theorem]{Remark}
\newtheorem{example}[theorem]{Example}

\newcommand\numberthis{\addtocounter{equation}{1}\tag{\theequation}}

\newcommand{\asto}{%
	\mathrel{\vbox{\offinterlineskip\ialign{%
				\hfil##\hfil\cr
				$\scriptstyle a.s.$\cr
				$\longrightarrow$\cr
}}}}

\newcommand{\dto}{%
	\mathrel{\vbox{\offinterlineskip\ialign{%
				\hfil##\hfil\cr
				$\scriptstyle d$\cr
				$\longrightarrow$\cr
}}}}

\newcommand{\vto}{%
	\mathrel{\vbox{\offinterlineskip\ialign{%
				\hfil##\hfil\cr
				$\scriptstyle v$\cr
				$\longrightarrow$\cr
}}}}

\newcommand{\rmd}{\mathrm{d}}
\newcommand{\toi}{\to\infty}

\newcommand{\EE}{\mathbb{E}}
\newcommand{\ZZ}{\mathbb{Z}}
\newcommand{\PP}{\mathbb{P}}
\newcommand{\NN}{\mathbb{N}}
\newcommand{\RR}{\mathbb{R}}
\newcommand{\MM}{\mathbb{M}}

\newcommand{\wid}{\widetilde}
\newcommand{\eind}{\stackrel{d}{=}}

\newcommand{\footremember}[2]{%
    \footnote{#2}
    \newcounter{#1}
    \setcounter{#1}{\value{footnote}}%
}

\title{On renewal theory for cluster processes}

\author{%
  Bojan Basrak\footremember{BB}{Department of Mathematics, Faculty of Science, University of Zagreb, Bijeni\v{c}ka 30, 10000 Zagreb, Croatia;
  \href{mailto:bbasrak@math.hr}{bbasrak@math.hr}}%
  \and Marina Dajakovi\'{c}\footremember{MD}{Faculty of Electrical Engineering, Mechanical Engineering and Naval Architecture, University of Split, Ru\dj era Bo\v{s}kovi\'ca 32, 21000 Split, Croatia; \href{mailto: marina.dajakovic@fesb.hr}{marina.dajakovic@fesb.hr} (corresponding author)}%
  }

\date{}

\begin{document}
	
\maketitle

\begin{abstract}{	
\noindent We prove several forms of renewal theorem  tailored to renewal processes with marks and clusters. In particular, for an i.i.d.\ sequence $(\xi_i,X_i)_{i \geq 0}$, where $\xi_0$ denotes a finite point process on $\RR$ and $X_0$ denotes a nonnegative random variable of finite mean, we consider the renewal sequence $T_i = X_0+\cdots + X_i$, $i \geq 0$, and corresponding renewal cluster process $ \xi(\cdot )=\sum_{i\geq0}\xi_i(\,\cdot -T_i)$.
Under mild assumptions on the distribution of $(\xi,X)$,   we show by coupling methods that the generalized versions of Blackwell's renewal theorem, key renewal theorem, extended renewal theorem and elementary renewal theorem still hold, even with dependence between $\xi_i$'s and $X_i$'s.}
\newline

\end{abstract}

\section{Introduction}

In its classical form,  renewal theory studies the steps of a renewal process made before or after some distant time $t$.
Renewal theorems have been shown to hold in various forms and very general settings, we refer to \cite{res92}, \cite{tho00} or unpublished monograph~\cite{als13} for general theory, and to \cite{ko08}, \cite{joc12} or \cite{iks16} for some interesting extensions. Here we prove several types of renewal theorem tailored to renewal processes with marks or clusters.  Random models of this type have been suggested repeatedly in applied literature, see for instance \cite{onof00} and \cite{mm10} for  applications to environmental sciences and insurance respectively. Suppose that at each step  $T_i = X_0+ X_1+\cdots + X_i$, $i \geq 0$, of a renewal process we observe multiple events scattered around it in a way which is possibly dependent on the last interarrival time. One of our main goals is to describe the distribution of those points beyond time $t$, for  $t\toi$.
The similar types of dependence between the steps of the renewal process and the shapes of the clusters we allow in the sequel have already been considered in different settings, see for instance \cite{mar15} in the context of c\`adl\`ag processes, or \cite{ag06}, but in the context of perturbed random walks.

The paper is organized as follows: in the next section we present a version of extended renewal theorem for renewal process with marks which will serve as the main tool in the rest of the article, and which seems of independent interest. 
Using it, in Section~\ref{sec:main}  we deduce 
versions of Blackwell's renewal theorem, key renewal theorem, extended renewal theorem and elementary renewal theorem even in the presence of dependent clusters. Note, some simple results of this type exist, but only for two special Poisson cluster processes, going back to \cite{lew64} and \cite{vere70}.
We finally discuss two examples, one of which is of a special Poissonian type, at the end of Section \ref{sec:main}.
The proofs are postponed to the Section~\ref{sec:proofs}.

\section{Extended renewal theorem for marked renewal process}

In this section, we introduce the notation and recall same basic facts about the vague convergence on the space of counting measures. For a general Polish space   \(\mathbb{S}\) let \(\mathcal{B}(\mathbb{S})\) be a Borel \(\sigma\)--algebra on \(\mathbb{S}\). In what follows, either \(\mathbb{S}=\RR\) or \(\mathbb{S}=\RR\times\MM\), where \(\MM\) is a Polish space for marks. We declare a set  \(B \) in \( \RR\)  or in \(  \RR \times \MM\) bounded if \(B \subseteq [-h,h]\) or \(B\subseteq [-h,h] \times \MM\), for some \(h>0\), respectively. 
Families of such bounded sets define  a {\em boundedness} on $\mathbb{S}$ in terminology used by \cite{bp19}.
Denote by  $M_p(\mathbb{S})$  the space of point measures  on  $\mathbb{S}$ which are finite on all bounded sets.
Let \(t\in\RR\) and \(m\in M_p(\mathbb{S})\). The shifted measure \(\theta_t m\) is defined by \(\theta_t m(A)=m(A+t)\), \(A\in\mathcal{B}(\mathbb{R})\) or \(\theta_t m(A\times B)=m(\{A+t\}\times B)\), \(A\in\mathcal{B}(\mathbb{R})\), \(B\in\mathcal{B}(\mathbb{\MM})\),  on \(M_p(\RR)\) or \(M_p(\RR\times\MM)\) respectively.
Recall that a sequence \(m_1,m_2,\ldots \in M_p(\mathbb{S})\) is said to converge vaguely to a measure \(m \in M_p(\mathbb{S})\), denoted by \(m_n \vto m\), if 
$
\int f \rmd m_n \to \int f \rmd m\,,
$
as \(n \toi\), for all continuous and bounded real-valued functions \(f\) on \(\mathbb{S}\) with bounded support. 
Convergence in distribution of point processes is always considered with respect to the corresponding vague topology and is denoted by \("\dto"\).

Consider i.i.d.\ pairs $(\wid W_i,\wid X_i)\,,{i\geq0} \in \mathbb{M} \times \RR_+$ with arbitrarily dependent \(\wid W_1\) and \(\wid X_1\), where \(\mathbb{M}\) is a general Polish space and \(\RR_+:=[0,\infty\rangle\). However, for notational convenience, and without loss of generality, we will assume in the sequel that $
\wid X_i = \varphi(\wid W_i)\,, i \geq0\,,
$
for some continuous function $\varphi : \mathbb{M}\to \RR_+$. If this is not the case, one can introduce the sequence \(\overline W_i=(\wid W_i, \wid X_i)\,, i\geq0\) on the space \(\mathbb{M}\times\mathbb{R}\) which is clearly Polish, and use the projection on the second coordinate as the mapping \(\varphi\). Furthermore, suppose that random variables
$
\wid X_i\,, i \geq0\,,
$
 satisfy $\mu  = \EE \wid X_1 \in \langle0,\infty\rangle$ and have a {\em nonarithmetic} distribution, i.e.\ their law is not concentrated on the set of the form $ h \ZZ$.
Define 
$
 \wid T_0 = 0 \text{ and }
 \wid T_i = \sum_{j=1}^i  \wid X_j\,,  i\geq1\,. 
$
One can form a two-sided renewal process with arrival times \(\cdots<\wid T_{-2}<\wid T_{-1}<\wid T_0 = 0<\wid T_1 <\wid T_2\cdots\), by extending the pure renewal process \((\wid T_i)_{i\geq0}\) on \(\RR_+\) in the obvious way: let \(\wid W_{-i}\) be an independent copy of \(\wid W_i\), for all \(i\geq1\), and define
$
\wid T_{-i} = - \sum_{j=-i+1}^{0} \wid  X_{j}\,, i \geq 1\,,
$
where \( \wid X_{-i} = \varphi(\wid W_{-i})\).
Consider the following point process
$
\tilde \eta  =  \sum_{i\in\ZZ} \delta_{(\wid T_i,\wid W_i)}\,.
$
By construction,
$
\theta_{\wid{T}_k}  \tilde \eta 
 \stackrel{d}{=} \tilde \eta\,, 
$
for all \(k\), hence $\tilde\eta$ is called {\em point stationary} marked renewal processes. 

Introduce next an additional random element \(W^*\) in $\mathbb{M}$ independent of the sequence \((\wid W_i)_{i\neq0}\) with  the distribution of $\wid W_1$ biased by \(\wid X_1\), i.e.
$\EE f(W^*)= \EE\Big[\wid X_1 f(\wid W_1)\Big] /\mu \,,
$
for all bounded measurable real-valued functions \(f\).
Denote $X^* = \varphi(W^*)$,
and let \(U\) denote a uniform random variable on \([0,1]\)  independent of $W^*, (\wid W_i)_{i\neq0} $. Define
$
 	T_0 = U X^* \mbox{ and }  	T_{-1} = - (1- U) X^*\,,
$
and then recursively
$
T_i = T_0 + \sum_{j=1}^{i} \wid X_j\,,
T_{-(i+1)}=T_{-1}-\sum_{j=-i}^{-1} \wid X_j\,,   i \geq 1
$
and set 
$
    W_0 = W^*\,, W_i = \wid W_i\,, \text{ for all }  i \not = 0\,. 
$
Then
$$
 	\eta  =  \sum_{i\in\ZZ} \delta_{ (T_i,  W_i)}
$$
is a {\em  stationary} marked renewal point process (cf.\ \cite{sig06}, Example 8.1.8).

Finally, we introduce a  {\em delayed  marked renewal process} on $\RR_+$ which is not assumed to be stationary in either sense. Let $(W'_0,X'_0)$ be a random element in $\mathbb{M}\times\RR_+$ which is independent of an i.i.d.\ \(\mathbb{M}\)-valued sequence
 $(W'_i)_{i \geq 1}$ with the property \(W'_i \eind \wid W_1\), for all \(i\geq1\). Denote  $ X'_i = \varphi( W'_i)$, for all $ i \geq 1$ and
 $T'_i = \sum_{j=0}^i X'_i$, for all $ i \geq 0$. Consider now 
 marked renewal  process 
 \[
 \eta'  =  \sum_{i=0}^\infty \delta_{ (T'_i,  W'_i)}\,.
 \] 
 We sometimes identify $\eta'$ with the sequence $(T'_i,  W'_i)_{i\geq0}$.

An  extended renewal theorem  for random walks appears as Theorem 12.7  of \cite{kal17}. We show that the corresponding result holds  for  marked renewal processes  as well. It is not difficult to see that one can use it to deduce the corresponding theorem in \cite{kal17} by taking excursions of random walk between two consecutive ladder points as marks. Actually, a special case of this was proved analytically in \cite{mar15}, see Lemma 3.1 therein. Related problems in renewal theory were studied earlier by \cite{imm17}. Our main result here  is that the  extended renewal theorem holds more generally. We base our probabilistic proof on coupling ideas, but postpone it until Section~\ref{sec:proofs}.

\begin{theorem}[Extended renewal theorem for marked renewal processes]
 	\label{lm0}
    Suppose \((W'_0,X'_0)\) is a random element in $\MM\times\RR_+$  independent of an i.i.d.\ sequence \((W'_i,X'_i)_{i\geq1}\) with values in \(\MM\times\RR_+\), where \(X'_1\) has a nonarithmetic distribution with a finite positive mean. Let \(\eta'\) be a corresponding delayed marked renewal process on \(\RR_+\) with a stationary version \(\eta\). Then
 	\[
 	\theta_t \eta' \dto \eta\,,
 	\] 	
 	as $t \toi$ with respect to the vague topology on $M_p(\RR \times \mathbb{M})$.
\end{theorem}

\section{Renewal theorems for cluster processes}\label{sec:main}

We henceforth assume
that \((X'_i)_{i\geq1}\) is a sequence of nonnegative random variables with a {\em nonarithmetic}  distribution and a finite  mean \(\mu >0\).
Furthermore, let \((\xi'_i)_{i\geq1}\) be a sequence of point processes on \(\RR\) such that the pairs \(W'_i=(\xi_i',X_i')\) are independent and identically distributed.  
Note that \(\xi'_1\) and \(X'_1\) are allowed to be dependent.
Assume further that a  random  pair $W'_0=(\xi_0',X'_0)$ also takes values in $M_p(\RR) \times \RR_+$ and is independent of the i.i.d.\ sequence \((W'_i)_{i\geq1}\). Let \((T'_i)_{i\geq0}\) be a renewal sequence with increments \(X'_i\), i.e. 
\(
T'_i=\sum_{j=0}^iX'_j\,,  i\geq0 .
\) 
The associated renewal cluster point process is given by
\begin{equation}\label{4-eq:def rcpp}
\xi'(B)=\sum_{i\geq0}\xi'_i(B-T'_i)\,,\quad  B\in\mathcal{B}(\mathbb{R})\,\,,
\end{equation}
where we assume that almost surely
\begin{equation}\label{eq:ex cond}
    \xi'(B)<\infty\,,
    \text{ for all bounded }
    B\in \mathcal{B}(\RR)\,.
\end{equation}
Following the terminology of \cite{bbk20}, we call \(\zeta'=\sum_{i\geq0}\delta_{T'_i}\) the parent process and \(\xi'_i\), \(i\geq0\), the descendant process.
Condition \eqref{eq:ex cond} is often encountered in the literature, see e.g.\ p. 245 in \cite{kal17} or relation (2.3.9) in \cite{bbk20}.
In Remark \ref{rmk:exist} we give one simple sufficient condition for it to hold.

Each $\xi'_i$ has a representation
\[
\xi'_i = \sum_{j=1}^{L'_i}\delta_{T'_{ij}}\,,
\]
where \(L'_i\) is a random variable with values in \(\{0,1,\dots,+\infty\}\) and $T'_{i,1},\ldots, T'_{i,L'_i}$ is a sequence of random variables with values in \(\RR\), see Theorem 2.18 in \cite{kal21}.
We will assume henceforth that \(L'_i\) is almost surely finite, for all \(i\). This assumption is a standard in the literature, see for instance p.\ 176 or Problem 6.3 in \cite{dvj03}. Specifically, it is satisfied for Bartlett-Lewis processes, cf. Section \ref{sec:examples}, as well as for the Hawkes processes.
One can identify
\begin{align}\label{4-eq:wi}
W'_i=(\xi_i',X_i') \quad \text{ with }\quad (L'_i, (T'_{ij})_{j},X'_i)\,,
\end{align}
for all \(i\geq0\).
Thus, we study a renewal cluster point process constructed in such a way that around each point \(T'_i\) of the parent process one observes cluster of $L'_i$ other points translated by random times \((T'_{ij})_{j}\). More precisely,
\[
\xi'(B)
=\sum_{i\geq0}\sum_{j=1}^{L'_i}\delta_{T'_{ij}}(B-T'_i)
=\sum_{i\geq0}\sum_{j=1}^{L'_i}\delta_{T'_i+T'_{ij}}(B)\,,
\]
for every \(B\in\mathcal{B}(\RR)\).
Although the points of the parent process \(\zeta'\) are often unobserved, it is possible to include them in the cluster point process, which will be done 
in Section~\ref{sec:examples}. One can simply set \(T'_{i0}=0,\) for every \(i\) and write 
\[
\xi'_i=\sum_{j=0}^{L'_i}\delta_{T'_{ij}}\quad \text{ and }\quad \xi'=\sum_{i\geq0}\sum_{j=0}^{L'_i}\delta_{T'_i+T'_{ij}}\,.
\]
In this case, the \(i-\)th cluster size is \(L'_i+1\).

Clearly one can regard \(W'_i\) in \eqref{4-eq:wi} as  random elements in Polish space \(\mathbb{M}=\mathbb{N}_0\times\mathbb{R}^{\NN}\times\RR_+\), for all \(i\geq0\). Hence, \(\sum_{i\geq0}\delta_{(T'_i,W'_i)}\) is a marked renewal process on \(\mathbb{R}_+\times\MM\) and we can, as in the previous section, construct its stationary version.
It now follows that the process
\[
\xi=\sum_{i\in\mathbb{Z}}\sum_{j=1}^{L_i}\delta_{T_i+T_{ij}}\,,
\] 
provided that it is finite a.s. on bounded sets, can be referred to as a {\em  stationary renewal cluster point process}.

\begin{proposition}\label{thm:mm of stat pp}
Mean measure of a stationary renewal point process \(\xi\) equals
\[
\EE\xi(B)=\frac{1}{\mu}\EE L_1 Leb(B)\,,
\]
for all \(B\in\mathcal{B}(\RR)\).
\end{proposition}

Application of the Campbell-Little-Mecke formula yields the statement of the proposition. Detailed proof is postponed to Section~\ref{sec:proofs}.

\begin{remark}\label{rmk:exist}
It follows immediately from the previous proposition that a finite mean cluster size is a simple sufficient condition for the existence of a stationary renewal cluster point process. However, this condition is in general not necessary, see \cite{dvj03}, p.\ 177. 
\end{remark}

It is useful in the sequel to introduce the following notation
 \[
    R'_i = \sup_{j \leq L'_i} |T'_{ij}|\,, \quad i\geq0\,.
 \]
Observe that $ R'_i \,, i \geq 1$ are i.i.d.\ nonnegative random variables.
Clearly they represent the distance between the farthest point in a cluster and the corresponding point in the parent process. Recall also that we assumed \(L'_i < \infty\) a.s. We note that without this constraint, one can not expect that the following renewal theorem  holds without additional restrictions. Intuitively, this is because clusters from the past could potentially continue to produce arrivals indefinitely, resulting in an ever-increasing number of arrivals following time \(t\).

\begin{theorem}[Extended renewal theorem for cluster processes]\label{thm:ert}
Assume that  \(\xi'\) is a renewal cluster point process with a stationary version \(\xi\). If \(\EE R'_1\) is finite, then
\begin{align*}
    \theta_t\xi'\dto\xi\,,
\end{align*}
that is
\begin{align*}
    \sum_{i\geq0}\sum_{j=1}^{L'_i}\delta_{T'_i-t+T'_{ij}}\dto\sum_{i\in\mathbb{Z}}\sum_{j=1}^{L_i}\delta_{T_i+T_{ij}}\,
\end{align*}
as \(t\to\infty\) with respect to the vague topology on \(M_p(\RR)\).
\end{theorem}

\begin{remark}
Let
\[
\cdots \leq S'_{-2} \leq S'_{-1} \leq 0\leq S'_{0}\leq S'_{1}\leq S'_{2}\leq\cdots
\]
denote the points of the renewal cluster point process \(\xi'\) in a nondecreasing order.
Generally, \(S'_k=T'_i+T'_{ij}\), for some \(i,j\).
For an arbitrary \(t>0\), set
\[
\sigma(t)=\inf\{k\in\mathbb{N}:S'_k>t\}\,,
\]
then
\[R(t)=S'_{\sigma(t)}-t\]
denotes the so-called forward recurrence time. In Section \ref{sec:examples}, for a specific cluster point process, we will obtain closed formula for the limiting distribution of the forward recurrence time as a consequence of Theorem \ref{thm:ert}.
\end{remark}

To deduce Blackwell's renewal theorem for renewal cluster point processes from the extended renewal theorem, we need a technical result involving uniform integrability. Recall that the random variables \(X_t\), \(t\in T\), are said to be uniformly integrable if
\[
\lim_{x\toi}\bigg(\sup_{t\in T}\EE\big[|X_t|\mathbbm{1}_{\{|X_t| > x\}}\big]\bigg)=0\,,
\]
where \(T\) is a nonempty index set.

\begin{assumption}[Finite moments assumption]\label{ass:fma}
	Suppose that random variables
	\(L'_0, L'_1, L'_1X'_1\) and \(L'_1R'_1\)  all have finite mean.
\end{assumption}

\begin{lemma}\label{thm:ui}
Let \(\xi'\) be a renewal cluster point process. Suppose that Assumption \ref{ass:fma} holds.
Then for all \(x>0\)
\[
\xi'\langle t,t+x]\,, \quad t\geq0\,, 
\]
are uniformly integrable random variables.
\end{lemma}

The proof of the lemma is postponed to Section~\ref{sec:proofs}. Together with Theorem \ref{thm:ert}, the lemma yields the following (we again refer to Section~\ref{sec:proofs} for a proof).

\begin{corollary}[Blackwell's renewal theorem for cluster processes]\label{cor:brt}
Let \(\xi'\) be a renewal cluster point process. Suppose that Assumption \ref{ass:fma} holds, then
\[
\EE \xi' \langle t,t+x] \to \frac{1}{\mu}\mathbb{E}L'_1x\,,
\]
as \(t \to \infty\), for all \(x>0\).
\end{corollary}

In the cluster point process framework, just like in the standard setting, Blackwell's renewal theorem implies the elementary renewal theorem (the proof is simple, see \cite{MDthesis}).

\begin{corollary}[Elementary renewal theorem for cluster processes] \label{thm:elem rnw thm}
Let \(\xi'\) be a renewal cluster point process. Suppose that 
	 Assumption \ref{ass:fma} holds, then
\[
\frac{\EE \xi' \langle 0,t]}{t} 
\to
\frac{1}{\mu}\mathbb{E}L'_1\,,
\]
as \(t \to \infty\).
\end{corollary}

Finally, it is natural to ask whether equivalence between Blackwell's renewal theorem and the key renewal theorem can be established in the cluster point process setting. The answer is affirmative and it is fairly easy to check that the two versions of renewal theorem are equivalent, as the proof essentially follows the classical argument.
We define the renewal function \(U_{\xi'}:\RR\to[0,\infty]\) by
\[
U_{\xi'}(t)=\EE \Bigg[ \sum_{i=0}^\infty\sum_{j=1}^{L'_i}\mathbbm{1}_{\{T'_i+T'_{ij}\leq t\}}
\Bigg]\,.
\]

\begin{lemma} \label{lem:UxiFin}
Let \(\xi'\) be a renewal cluster point process and suppose that Assumption \ref{ass:fma} holds. Then the renewal function \(U_{\xi'}(t)\) takes finite value for all \(t\in\RR\).
\end{lemma}

The  lemma is proved in Section~\ref{sec:proofs}. Together with the dominated convergence theorem, it yields the right-continuity of the renewal function \(U_{\xi'}\). As it is obviously nondecreasing, we can associate with it a unique (renewal) measure \(\mu_{U_{\xi'}}\) on \(\RR\) such that \(\mu_{U_{\xi'}}\langle-\infty,t]=U_{\xi'}(t)\), for all \(t\). 
It is common in the literature to use the same notation for the renewal function and this measure, since the context prevents the confusion. Hence, in the sequel we will denote both objects by \(U_{\xi'}\).

\begin{corollary}[Key renewal theorem for cluster processes]
	\label{cor:KRT}
Let \(\xi'\) be a renewal cluster point process. Suppose that Assumption \ref{ass:fma} holds and that \(g:\mathbb{R_+}\to\mathbb{R_+}\) is a directly Riemann integrable function. Then
	\[
	\lim_{t\to\infty}\int_0^t g(t-y)\rmd U_{\xi'}(y)
	=\frac{1}{\mu}\EE L'_1 \int_{0}^\infty g(y) \rmd y\,.
	\]
\end{corollary}

The proof of the corollary can be found in Section~\ref{sec:proofs}. It shows that Blackwell's renewal theorem implies the key renewal theorem under Assumption \ref{ass:fma}. The reverse implication is of course straightforward if one sets \(g=\mathbbm{1}_{[0,x\rangle}\), for any fixed  \(x\geq0\).

\subsection{Examples}\label{sec:examples}

For an illustration we consider two processes. In the first example, we study a well known Poisson cluster process, and in the second, a renewal cluster point process where both the number and position of points within the cluster depend on the last interarrival time.


\begin{example}[Bartlett-Lewis process]
Suppose that \(T'_i, i\geq1\) are the points of a homogeneous Poisson process on \(\mathbb{R}_+\) with intensity \(\lambda\) and \(Y'_{ij}, {i\geq1,j\geq1}\) are i.i.d.\ nonnegative random variables with distribution \(F\) independent of \(T'_i\), for all \(i\). Let \(L'_i, {i\geq1}\) be i.i.d.\ nonnegative integer-valued random variables independent of \(T'_i\) and \(Y'_{ij}\), for all \(i,j\). Define
\(T'_{ik}=\sum_{j=1}^{k}Y'_{ij}\) and set \(T'_{i0}=0\), for every \(i\geq1\).
Bartlett-Lewis point process is given by
\[
\sum_{i\geq1}\sum_{j=0}^{L'_i}\delta_{T'_i+T'_{ij}}\,.
\]
Let \(x>0\). Note that the points of the parent process are now included in the cluster point process.
Hence, a simple modification of Corollary \ref{cor:brt} yields
\begin{equation}\label{4-eq:mean bl}
\lim_{t\to\infty} \mathbb{E} \Bigg[ \sum_{i\geq1} \sum_{j=0}^{L'_i} \delta_{T'_i+T'_{ij}} \Bigg] \langle t,t+x]
=\frac{1}{\mu} \EE [L'_1+1] x
=\lambda \mathbb{E}[L'_1+1] x\,,
\end{equation}
under the Assumption \ref{ass:fma}.
Furthermore, if \(L'_1\)  and \(Y'_{11}\) are integrable, then by Theorem \ref{thm:ert} and Theorem 4.11 in \cite{kal17}
\begin{gather*}
    \lim_{t\to+\infty}
    \mathbb{P} \Bigg( \sum_{i\geq1} \sum_{j=0}^{L'_i} \delta_{T'_i+T'_{ij}} \langle t,t+x]=0 \Bigg)
    =\mathbb{P} \Bigg( \sum_{i\in\mathbb{Z}} \sum_{j=0}^{L_i} \delta_{T_i+T_{ij}} \langle 0,x]=0 \Bigg)\\
    =\exp \bigg\{-\lambda \bigg(x+\mathbb{E}{L'_1} \int_0^x \mathbb{P}(Y'_{11}>y) \rmd y \bigg)\bigg\}\,,
    \numberthis\label{4-eq:R bl}
\end{gather*}
where the proof of the last equality can be found in Section 5 of \cite{fay06}, see also Example 6.3 (b) in \cite{dvj03}. Hence,
\[
\lim_{t\to+\infty}\mathbb{P}(R(t)\leq x)
=1-\exp\bigg\{-\lambda \bigg( x+\mathbb{E}{L'_1} \int_0^x \mathbb{P} (Y'_{11}>y) \rmd y \bigg)\bigg\}\,.
\]
\end{example}

Equations \eqref{4-eq:mean bl} and \eqref{4-eq:R bl}, which we obtained as a consequence of the Theorem \ref{thm:ert} and Corollary \ref{cor:brt}, can be found in \cite{lew69}, Theorem 2.3 and \cite{lew64}, equation (4.3.5).

\begin{example}
Suppose that \((T'_i)_{i\geq0}\) is a pure renewal process with interarrivals distributed uniformly on \(\langle0,5\rangle\). Let \(L'_0\equiv0\) and suppose that both \(L'_i\) and \(T'_{ij}\) depend on the last interarrival \(X'_i\), for all \(i\geq1\). Precisely, let \(L'_{i1}\) and \(L'_{i2}\) be i.i.d.\ Poisson random variables with parameters \(0.5\) and \(5\), respectively and let \(Y'_{ij}\) be i.i.d.\ random variables with standard normal distribution, for all \(i,j\geq1\). Then, we set
\[
L'_i=L'_{i1}\mathbbm{1}_{\{X'_i>1\}}+L'_{i2}\mathbbm{1}_{\{X'_i\leq1\}}
\quad\text{ and }\quad
T'_{ij}=X'_i+Y'_{ij}\,,
\]
for all \(i,j\geq1\).
Renewal cluster point process is given by
\[
\xi'=\sum_{i\geq0}\sum_{j=1}^{L'_i}\delta_{T'_i+T'_{ij}}\,.
\]
By direct calculation one easily checks the conditions of Corollaries \ref{cor:brt} and \ref{thm:elem rnw thm}, hence
\[
\lim_{t\toi}\EE\xi'\langle t, t+x]
=\frac{1}{\EE X'_1} \EE L'_1x
=0.4\cdot1.4x
=0.56 x\,,\quad \text{ for all } x>0
\]
and
\[
\lim_{t\toi}\frac{\EE\xi'\langle 0, t]}{t}
=\frac{1}{\EE X'_1} \EE L'_1
=0.56\,.
\]
For Monte Carlo simulations supporting these results, we refer to \cite{MDthesis}.
\end{example}

\section{Proofs}\label{sec:proofs}

\subsection{Proof of Theorem~\ref{lm0}}

Recall a random variable \(\vartheta\) is said to be Rademacher if 
\(\PP(\vartheta = 1) = \PP(\vartheta = -1)  = 1/2\).
A sequence of i.i.d. Rademacher variables is called a Rademacher sequence. The following lemma seems to be known, a rigorous proof can be found in \cite{MDthesis}.
\begin{lemma}\label{lm:1}
	Suppose that $\beta$ is a random element in some Polish space $\MM$ independent of a Rademacher sequence $(\vartheta_k)_{k \in \NN}$ and let $\mathcal{F}_n = \sigma\{ \beta, \vartheta_1,  \ldots ,\vartheta _n \}$. Assume that $\tau$ is $\{\mathcal{F}_n\}$--stopping time, then the sequence
	\[
	(\beta , \vartheta_1,  \ldots ,\vartheta _\tau,  - \vartheta _{\tau+1},- \vartheta _{\tau+2},  \ldots )
	\]
	has the same distribution as $ (\beta , \vartheta_1,\vartheta _2, \vartheta _3, \ldots )$\,.
\end{lemma}

\begin{lemma}\label{4-lem inc}
Let \(\MM\) be an arbitrary Polish space. Assume that for an arbitrary nonnegative continuous real-valued function \(f\) on \(\MM\) and a sequence \((t_0^{(n)},(w_i^{(n)})_{i\geq0})_{n\geq1}\) in \(\RR_+\times\MM^{\NN_0}\) the following conditions hold:
\begin{enumerate}[(i)]
    \item \(
    t_0^{(n)}\to t_0\text { in } \RR_+ \text{ and } w_i^{(n)}\to w_i\text{ in } \MM \,,\text{ for all }i\geq0\,,\text{ as }n\toi
    \)\,,
    \item \(
    \sum_{j=1}^i f(w_j^{(n)})\toi \text{ and }
  \sum_{j=1}^i f(w_j)\toi\,, \text{ for all }n\in\NN\,,\text{ as } i\toi
  \)\,.
\end{enumerate}
\noindent Then
\[
\sum_{i=0}^\infty\delta_{(t_0^{(n)}+\sum_{j=1}^i f(w_j^{(n)}),w_i^{(n)})}
\vto
\sum_{i=0}^\infty\delta_{(t_0+\sum_{j=1}^i f(w_j),w_i)}\,,
\]
as \(n\toi\).
\end{lemma}

\begin{proof}
Denote
\[
m_n=\sum_{i=0}^\infty\delta_{(t_0^{(n)}+\sum_{j=1}^i f(w_j^{(n)}),w_i^{(n)})}\,,\text{ for all }n\in\NN \quad\text{ and }\quad 
m=\sum_{i=0}^\infty\delta_{(t_0+\sum_{j=1}^i f(w_j),w_i)}\,.
\]
From assumption (ii),
it follows that integer-valued measures \(m_n\) and \(m\) are locally finite on \(\RR\times\MM\). 
Furthermore, we conclude from (i) and (ii)
that for an arbitrary \(h>0\) such that \(m(\partial([0,h]\times\MM))=0\), there exist \(k_0, n_0\in \NN\) such that for all \(n\geq n_0\)
\[
m_n\bigg|_{[0,h]\times\MM}=\sum_{i=0}^{k_0}\delta_{(t_0^{(n)}+\sum_{j=1}^i f(w_j^{(n)}),w_i^{(n)})}
\quad\text{ and }\quad
m\bigg|_{[0,h]\times\MM}=\sum_{i=0}^{k_0}\delta_{(t_0+\sum_{j=1}^i f(w_j),w_i)}\,.
\]
Moreover, by (i)
\[
t_0^{(n)}+\sum_{j=1}^i f(w_j^{(n)})\to t_0+\sum_{j=1}^i f(w_j) \text{ in } \RR_+
\quad\text{ and }\quad
w_i^{(n)}\to w_i\text{ in } \MM\,,
\]
for all \(i=0,1,\dots,k_0\), since \( f\) is continuous.
Thus, by Proposition 2.8 in \cite{bp19}
\[
m_n\vto m\,,
\]
as \(n\toi\).
\end{proof}

\begin{proof} [Proof of Theorem~\ref{lm0}]
	Like several other proofs of different types of renewal theorem, we base our argument  on the coupling method. In the initial steps, it follows a coupling of the type used by \cite{tho00} in sections 2.7 and 2.8 and \cite{kal17} in the proof of Theorem 12.7.

	Consider an i.i.d.\ sequence  $(\wid W_i, \wid X_i)_{i \geq1}$ with the same distribution as $(W_i,X_i)_{i\geq1}$ and an independent Rademacher sequence $(\vartheta_i)_{i\geq1}$.  
	
	\noindent
	\textit{Step 1.} We first  construct copies of $\eta$ and $\eta'$ on the same probability space. 
	Assume that $(W_0,X_0)$ and $(W'_0,X'_0)$ have the desired distribution, but are independent of the sequences $(\wid W_i, \wid X_i)_{i\geq1}$ and $(\vartheta_i)_{i\geq1}$. Set $K_0=K_0''=0$ and then for $i \geq 1$
	\[
	K_i = \inf \{ k> K_{i-1} \,:\, \vartheta_k = 1 \}
	\quad \mbox{ and }\quad
	K''_i = \inf \{ k> K''_{i-1} \,:\, \vartheta_k = -1 \} \,.
	\]
	Alternatively, one can view $K_i$ and $K''_i$ as the indices of the $i$-th $1$ and $-1$, respectively, in the sequence $(\vartheta_i)_{i\geq1}$.
	Let	
	$T_0= X_0$, $T''_0 = X'_0$ and \(W''_0=W'_0\). Then, for $i \geq 1$ let
	\begin{align}
		T_i &=  T_0 + \sum_{j=1}^{K_i} \wid X_j \mathbbm{1}_{\{\vartheta_j=1\}} 
		= T_0 + \sum_{j=1}^{i} \wid X_{K_j} = : T_{i-1}+Y_i \,, 
		\quad \mbox{ and } \quad  W_i = \wid W_{K_i}\,,\label{eq:TpY}   \\ 
		T''_i &= T''_0 + \sum_{j=1}^{K''_i} \wid X_j \mathbbm{1}_{\{\vartheta_j=-1 \}}
		= T''_0 + \sum_{j=1}^{i} \wid X_{K''_j}\,,
		\quad \mbox{ and } \quad   W''_i = \wid W_{K''_i}\,. \label{eq:Tdvoc}
	\end{align}
	Note that $Y_i  = \wid X_{K_i}$ in \eqref{eq:TpY} are i.i.d., for all \(i\geq1\).
	As explained above, one can extend the sequence $( T_i,  W_i)_{i\geq 0}$ to a sequence indexed over $\ZZ$ so that the point process
	$\eta = \sum_{i \in \ZZ} \delta_{ (T_i,  W_i)}$ becomes stationary.
	It is immediate then that $\eta $ and 
	\[
	\eta''  =  \sum_{i\geq0} \delta_{ (T''_i,  W''_i)}\,.
	\]
	have the same distribution as the $\eta$ and $\eta'$ in the statement of the theorem. 
	
	\noindent
	\textit{Step 2.} Set 
	\[
	L_i = | \{  k \leq i : \vartheta_k = 1 \}|
	\quad \text{ and } \quad
	L'_i = i - L_i\,,
	\quad i\geq0\,,
	\]
	where \(\vert\cdot\vert\) denotes cardinality of a set.
	By \eqref{eq:TpY} and \eqref{eq:Tdvoc}, the difference between the arrival times $T_{L_i}$ and $T''_{L'_i}$
	equals 
	\[
	V_i = T_0 - T''_0 + \sum_{j=1}^i \vartheta_j  \wid X_j\,,\quad i \geq 0\,.
	\]
	In particular $(V_i)_{i\geq0}$ is a random walk with nonarithmetic and symmetric steps. By the Chung-Fuchs theorem, see Theorem 4 in \cite{cf51}, \((V_i)_{i\geq0}\) is recurrent and for all \(x\in\RR\) and \(\varepsilon>0\)
	\[
	\PP(|V_i-x|<\varepsilon \text{ i.o.})=1\,.
	\]
	Therefore, for any $\varepsilon>0$, with probability 1
	\[
	\tau = \inf \{i : V_i \in [0,\varepsilon\rangle \} < \infty \,.
	\]
	
	\noindent
	\textit{Step 3.} 
	Let now
	\begin{align*}
		T'_i &= T''_i \quad \text{ and } \quad W_i' = W_i''\,, \hspace{5.4cm}\quad \mbox{ if } i \leq L'_\tau\ \\
		T'_i &= T'_{i-1} + Y_{ k+ L_\tau } = T'_{i-1} +\wid X_{K_{ k+ L_\tau} }  \quad \text{ and } \quad   W_i' = W_{k+L_\tau}\,, \quad \mbox{ if } i= L'_\tau+k\,, k>0\,.
	\end{align*}
	Then in particular, for any $k\geq 0$
	\[
	T_{k+L_\tau}- T'_{k+L'_\tau} \in [0, \varepsilon\rangle \quad \mbox{ and } \quad  W_{k+L_\tau} = W'_{k+L'_\tau}\,.
	\]
	That is, after coming $\varepsilon$--close at times $L_\tau$ and $L'_\tau$, the marked point processes $\eta$ and
	\[ 
	\eta'  =  \sum_{i\geq0} \delta_{ (T'_i,  W'_i)}\,
	\]
	stay  $\varepsilon$--close in time with exactly the same marks indefinitely.
	Denote now by $(\wid \vartheta_i)_{i\geq1}$  Rademacher sequence
	\[
	\vartheta_1,  \vartheta_2,  \vartheta_3,  \ldots,  \vartheta_\tau, -  \vartheta_{\tau+1},
	- \vartheta_{\tau+2},- \vartheta_{\tau+3},\ldots
	\]
	Applying Lemma~\ref{lm:1} to the sequence $\beta =(\wid W_i, \wid X_i)_{i\geq1}$, the sequence $(\vartheta_i)_{i\geq1}$ and stopping time $\tau$, we see that \((\beta,\wid\vartheta_1,\wid\vartheta_2,\dots)\) has  the same distribution as \((\beta,\vartheta_1,\vartheta_2,\dots)\). Let $K'_0 =0$ and $K'_i = \inf \{ k> K'_{i-1} \,:\, \wid \vartheta_k = -1 \}$, \(i\geq1\).
	Since
	\begin{align*}
		T'_i &= T'_0 + \sum_{j=1}^{K'_i} \wid X_j \mathbbm{1}_{\{\wid \vartheta_j=-1\}}
		\quad\mbox{ and }\quad  
		W'_i = \wid W_{K'_i}\,,
	\end{align*}
	a quick comparison with \eqref{eq:Tdvoc} shows that 
	$ \eta'$ has the same distribution as in the statement of the theorem.
	
	\noindent
	\textit{Step 4.} 
	Observe that $V_{\varepsilon} := T_{L_\tau} \vee T''_{L'_\tau} < \infty $ a.s.
	Denote for $t >0$ 
	\[
	\sigma(t) = \inf \{ k \geq  0: T_k > t\} \quad \mbox{ and } \quad
	\sigma'(t) = \inf \{ k \geq 0  : T'_k > t\}\,.
	\]
	For arbitrary $u \geq 0$ and a measurable set $A_0 \subseteq \MM$
	we have
	\begin{align}
			&\PP \big(  T_{\sigma(t)} -t > u+\varepsilon,\,  W_{\sigma(t)}\in A_0 
			\big)
			-
			\PP(V_\varepsilon > t) \nonumber\\
		& \hspace{0.5cm} \leq \PP \big(  T'_{\sigma'(t)} -t > u,\,  W'_{\sigma'(t)}\in A_0 
		\big) \label{eq:PVc} \\ 
		& \hspace{1cm}\leq 
		\PP \big(  T_{\sigma(t)} -t \in [0, \varepsilon\rangle \big) +
		\PP \big(  T_{\sigma(t)} -t > u,\,  W_{\sigma(t)}\in A_0
		\big)
		+
		\PP(V_{\varepsilon} > t)\,. \nonumber	
	\end{align}
	Clearly, $\PP(V_{\varepsilon} > t) \to 0$ as $t \toi$. Moreover, by the stationarity of the process $\eta$ 
	\[
	\PP \big(  T_{\sigma(t)} -t > u,\,  W_{\sigma(t)}\in A_0
	\big) =
	\PP \big(  T_0 > u,\,  W_{0}\in A_0
	\big)\,.
	\]
	Observe further that the function $u \mapsto \PP (  T_0 \leq u,\,  W_{0}\in A_0
	)$ is continuous, for any choice of $A_0. $ 
	Therefore, letting first $t \toi$ and then $\varepsilon \to 0$ on the left and right hand side of \eqref{eq:PVc},
	yields 
	\begin{equation}\label{eq:TWconv}
		\Big(T'_{\sigma'(t)} -t ,  W'_{\sigma'(t)}\Big)
		\dto  (T_{0} , W_{0} )\,,
	\end{equation}
	as \(t\toi\).
	On the other hand, observe that
	\begin{equation}\label{eq:WstrongMark}
		\Big( W'_{\sigma'(t)+i} \Big)_{i \geq 1} \eind (W_i)_{i\geq 1}
	\end{equation}
	with the left hand side independent of $\Big(T'_{\sigma'(t)} -t,  W'_{\sigma'(t)}\Big)$ and the right hand side independent of \((T_0,W_0)\).
	
	\noindent
	\textit{Step 5.} Because $\RR\times \mathbb{M}$ is Polish space, there exists a  probability space where one can find random elements
	$(\hat T'_{\sigma'(t)} -t,  \hat W_{\sigma'(t)} , (\hat W_{\sigma'(t)+i})_{i\geq 1} ),\  t >0$,
	and $ (\hat T_{0}, \hat W_{0}, (\hat W_i)_{i\geq 1}  )$ which have the same joint distribution as the random elements
	in \eqref{eq:TWconv} and \eqref{eq:WstrongMark}, but satisfy
	\[
	\Big(\hat T'_{\sigma'(t)} -t ,  (\hat W_{\sigma'(t)+i})_{i\geq0}\Big)
	\asto  \Big(\hat T_{0}, (\hat W_i)_{i\geq0} \Big)\,,
	\]
	as \(t \toi\), in the product topology on $\RR\times \MM^{\NN_0}$.
	Moreover, by the strong law of large numbers (see e.g.\ Theorem 2.4.1 in \cite{dur10})
	\[
	\sum_{j=1}^i\varphi(\hat W_{\sigma'(t)+j})=\sum_{j=1}^i \hat X_{\sigma'(t)+j}
	\quad\text{and}\quad
	\sum_{j=1}^i\varphi(\hat W_{j})=\sum_{j=1}^i\hat X_{j}
	\]
	tend to $\infty$ almost surely, as $i \toi$. 
	Observe that by Lemma~\ref{4-lem inc}  
	\begin{align*}
	\theta_t \eta' 
	&\eind
	\sum_{i = 0}^\infty \delta_{\big(\hat T'_{\sigma'(t)} -t+ \sum_{j=1}^i \varphi\big(\hat W'_{\sigma'(t)+j}\big), \hat W'_{\sigma'(t)+i}\big) }
	\\
	&\hspace{3.5cm}\asto 
	\sum_{i = 0}^\infty \delta_{\big(\hat T_0 + \sum_{j=1}^i \varphi\big(\hat W_{j}\big), \hat W_{i}\big)} 
	\eind \eta \bigg|_{[0,\infty\rangle\times \mathbb{M}} \,,
	\end{align*}
	as \(t\toi\).
	This now yields the statement of the theorem on the state space
	$[0,\infty\rangle\times \mathbb{M}$.

	\noindent
	\textit{Step 6.}
	To extend convergence in $M_p([0,\infty\rangle \times \mathbb{M})$ to the  convergence in distribution
	with respect to the vague topology on $M_p(\RR \times \mathbb{M})$, consider an arbitrary 
	continuous bounded function $f:\RR \times \mathbb{M} \to \RR_+$ with a support on  a set of the form $[-h,h] \times \mathbb{M}$, for some $ h>0$. Introduce $f^{h}(u,w) = f(u-h,w)$ which is clearly continuous bounded and supported on $[0,2h] \times \mathbb{M}$. Now for any fixed $h>0$ as $t \toi$
	\begin{align*}
	    \EE \exp &\big( - \theta_t \eta' (f) \big)
		= \EE \exp \Bigg( - \sum_{i \geq -\sigma'(t)}   f \big(  T'_{\sigma'(t)+i} -t,    W'_{\sigma'(t)+i}\big)  \Bigg)\\
		&= \EE \exp \Bigg( - \sum_{i \geq -\sigma'(t)}   f^{h}
		\big(  T'_{\sigma'(t)+i} -(t-h),    W'_{\sigma'(t)+i}\big)  \Bigg)\\
		&= \EE \exp \Bigg( - \sum_{i \geq0}   f^{h}
		\big(  T'_{\sigma'(t-h)+i} -(t-h),    W'_{\sigma'(t-h)+i}\big)  \Bigg)\\
		&= \EE \exp \big( -\theta_{t-h} \eta' (f^{h})\big) 
		\longrightarrow \EE \exp \big( - \eta (f^{h})\big)\\
		&= \EE \exp \Bigg( - \sum_{i \geq0}   f^h
		\big(  T_{i},    W_{i}\big)  \Bigg)
		= \EE \exp \Bigg( - \sum_{i \in\ZZ}   f^h
		\big(  T_{i},    W_{i}\big)  \Bigg)\\
		&= \EE \exp \Bigg( - \sum_{i \in \ZZ}   f
		\big(  T_{i}-h,    W_{i}\big)  \Bigg)
		= \EE  \exp \big(- \eta  (f) \big) \,,
	\end{align*}
	where the last equality follows by the stationarity of the process $\eta$.
	Since $f$ was arbitrary, by Theorem 4.11 in \cite{kal17}, this concludes the proof.  
	
\end{proof}

\subsection{Proof of Proposition~\ref{thm:mm of stat pp}}
\begin{proof}
First, we introduce the notation. Let \(B+x=\{b+x:b\in B\}\), for all \(B\in\mathcal{B}(\RR)\) and \(x\in\RR\).
\begin{align*}
    \EE\xi (B)
    &=\EE \Bigg[\sum_{i\in\ZZ} \xi_i(B-T_i)\Bigg]
    =\EE  \Bigg[\sum_{i\in\ZZ} f_B (T_i, \xi_i)\Bigg]
    =\frac{1}{\mu} \int_\RR \EE [f_B(y,\xi'_1)] \rmd y\\
    &=\frac{1}{\mu}\EE \bigg[ \int_\RR \xi'_1(B-y)\rmd y\bigg]
    =\frac{1}{\mu}\EE  \bigg[ \int_\RR \int_\RR \mathbbm{1}_{\{s\in B-y\}} \rmd\xi'_1( s) \rmd y \bigg]\\
    &=\frac{1}{\mu}\EE \bigg[ \int_\RR \int_\RR \mathbbm{1}_{\{y\in B-s\}} \rmd y  \rmd \xi'_1(s) \bigg]
    =\frac{1}{\mu}\EE \bigg[ \int_\RR Leb(B) \rmd\xi'_1( s) \bigg]\\
    &=\frac{1}{\mu} \EE\xi'_1(\RR) Leb(B) 
    =\frac{1}{\mu} \EE L'_1 Leb(B)\,,
\end{align*}
where the third equality follows by the Campbell-Little-Mecke formula (see e.g.\ relation (1.2.19) in \cite{bb03}) applied to a stationary marked point process \(\sum_{i\in\ZZ}\delta_{(T_i, \xi_i)}\) for a nonnegative measurable function \(f_B:\RR\times M_p(\RR)\to \RR \) given by \(f_B(y,m)=m(B-y)\). Note that \(\EE\big[\sum_{i\in\ZZ}\delta_{T_i}[0,1]\big]=\frac{1}{\mu}\), see e.g.\ \cite{res92} p.\ 215. By construction, \(X_1\stackrel{d}{=}X'_1\) and \(L_1\stackrel{d}{=}L'_1\), which completes the proof.
\end{proof}

\subsection{Proof of Theorem~\ref{thm:ert}}
Let in the sequel \(f_+:=\max(f,0)\) and \(f_-:=-\min(f,0)\) denote the positive and negative part of a real-valued function \(f\).

\begin{proof}
	\textit{Step 1.} Note that the mark space \(\mathbb{M}=\mathbb{N}_0\times\mathbb{R}^{\NN}\times[0,\infty\rangle\) of the corresponding marked renewal point processes is again Polish, hence Theorem~\ref{lm0} can be applied, yielding
	\[
	\sum_{i\geq0}\delta_{(T'_i-t,W'_i)}\dto\sum_{i\in\mathbb{Z}}\delta_{(T_i,W_i)}\,,
	\]
	as \(t\to\infty\) in \(M_p(\mathbb{R}\times\mathbb{M})\). By the definition of convergence in distribution,
	\begin{equation}\label{step1}
		\sum_{i\geq0}\delta_{(T'_i-t,L'_i,(T'_{ij})_{j\leq L'_i})}\dto\sum_{i\in\mathbb{Z}}\delta_{(T_i,L_i,(T_{ij})_{j\leq L_i})}\,,
	\end{equation}
	as \(t\to\infty\) in \(M_p(\mathbb{R}\times\mathbb{M}^*)\), where \(\mathbb{M}^*=\mathbb{N}_0\times\mathbb{R}^\NN\). Here and in what follows, all indices \(j\) are assumed to be positive integers.

	\noindent\textit{Step 2.}
	Fix \(K>0\) and consider the mapping \(T_{K}:M_p(\mathbb{R}\times\mathbb{M}^*)\to M_p(\mathbb{R}\times \mathbb{R})\) given by
	\[
	T_{K}\Bigg(\sum_i\delta_{(x_i,y_i,(z_{ij})_{j\leq y_i})}\Bigg)
	=
	\sum_i\sum_{j=1}^{y_i}\delta_{(x_i,z_{ij})}\mathbbm{1}_{\{|x_i|\leq K\}}\,.
	\]
	It is well defined, i.e. for all \( m\in M_p(\mathbb{R}\times \mathbb{M}^*)\), \(T_{K}(m)\) is a locally finite point measure on \(\mathbb{R}\times \mathbb{R}\). Denote 
	\[
	N_K=\{m\in M_p(\mathbb{R}\times \mathbb{M}^*):m(\{\pm K\}\times \mathbb{M}^*)=0\}\,.
	\] 
	We will show that \(T_{K}\) is continuous on \(N_K\). Let \(m_1,m_2,\ldots \in M_p(\mathbb{R}\times\mathbb{M}^*)\), \(m\in N_K\) and \(m_n\vto m\). Then, by Proposition 2.8 in \cite{bp19}, there exist integers \(n_0, P\) and a labeling of the points of \(m\) and \(m_n\), \(n\geq n_0\) in \([-K,K] \times \mathbb{M}^*\) such that
	\[
	m_n\big|_{[-K,K]\times \mathbb{M}^*}= \sum_{i=1}^P\delta_{(x_i^{(n)},y_i^{(n)},(z_{ij}^{(n)})_{j\leq y_i^{(n)}})}\,,
	\] 
	\[
	m\big|_{[-K,K]\times \mathbb{M}^*}=\sum_{i=1}^P\delta_{(x_i,y_i,(z_{ij})_{j \leq y_i})}\
	\]
	and for all \(i=1,2,\ldots, P\)
	\[
	x_i^{(n)}\to x_i\,, \quad y_i^{(n)} \to y_i\,, \quad (z_{ij}^{(n)})_{j\leq y_i^{(n)}}\to (z_{ij})_{j\leq y_i}\,,
	\] as \(n\to\infty\). Observe that
	\[
	T_{K}(m_n)=\sum_{i=1}^P\sum_{j=1}^{y_i^{(n)}} \delta_{(x_i^{(n)},z_{ij}^{(n)})}\,, \quad \text{ for all } n\geq n_0\,,
	\]
	and
	\[
	T_{K}(m)=\sum_{i=1}^P\sum_{j=1}^{y_i}\delta_{(x_i,z_{ij})}\,,
	\]
	have only finitely many terms, 
	hence it holds that \(T_{K}(m_n)\to T_{K}(m)\), as \(n \to \infty\). Finally, an application of the continuous mapping theorem (see e.g.\ \cite{bil68}, Corollary 1, p.\ 31) to  (\ref{step1}) yields
	\[
	T_{K}\bigg(\sum_{i\geq0}\delta_{(T'_i-t,L'_i,(T'_{ij})_{j\leq L'_i})}\bigg)
	\dto T_{K}\bigg(\sum_{i\in\mathbb{Z}}\delta_{(T_i,L_i,(T_{ij})_{j\leq L_i})}\bigg)\,,
	\]
	thus
	\begin{equation}\label{step2}
		\sum_{i\geq0}\sum_{j=1}^{L'_i}\delta_{(T'_i-t,T'_{ij})}\bigg|_{[-K,K]\times\mathbb{R}}\dto
		\sum_{i\in\mathbb{Z}}\sum_{j=1}^{L_i}\delta_{(T_i,T_{ij})}\bigg|_{[-K,K]\times\mathbb{R}}\,,
	\end{equation}
	as \(t \to \infty\) in \(M_p(\mathbb{R}\times \mathbb{R})\). 
	Indeed, 
from Proposition 8.2 in \cite{lp17}, it follows that
	\begin{align*}
		\mathbb{P}\Bigg(\sum_{i\in\mathbb{Z}}\delta_{(T_i,L_i,(T_{ij})_{j\leq L_i})}\in N_K\Bigg)
		=\mathbb{P}\Bigg(\sum_{i\in\mathbb{Z}}\delta_{T_i}(\{\pm K\})=0\Bigg)
		=1\,.
	\end{align*}

	\noindent\textit{Step 3.}
	An application of the continuous mapping theorem to (\ref{step2}) yields
	\[
	\sum_{i\geq0}\sum_{j=1}^{L'_i}\mathbbm{1}_{\{|T'_i-t|\leq K\}}\delta_{T'_i-t+T'_{ij}}\dto
	\sum_{i\in\mathbb{Z}}\sum_{j=1}^{L_i}\mathbbm{1}_{\{|T_i|\leq K\}}\delta_{T_i+T_{ij}}\,,
	\]
	as \(t \to \infty\) in \(M_p(\RR)\).

	Taking the point process on the right hand side of the above relation and letting \(K\to\infty\), one obtains
	\begin{align}
		\sum_{i\in\mathbb{Z}}\sum_{j=1}^{L_i}\mathbbm{1}_{\{|T_i|\leq K\}}\delta_{T_i+T_{ij}}\dto
		\sum_{i\in\mathbb{Z}}\sum_{j=1}^{L_i}\delta_{T_i+T_{ij}}\,,\numberthis\label{4-eq:k toi}
	\end{align}
	in \(M_p(\mathbb{R})\). Indeed, let \(f:\RR \to \RR_+\) be an arbitrary continuous bounded function with a bounded support. Then, by the monotone convergence theorem,
	\[
	\sum_{i\in\mathbb{Z}}\sum_{j=1}^{L_i}\mathbbm{1}_{\{|T_i|\leq K\}} f(T_i+T_{ij}) \asto 
	\sum_{i\in\mathbb{Z}}\sum_{j=1}^{L_i}f(T_i+T_{ij})\,,
	\]
	as \(K\toi\). Now Theorem 4.11 in \cite{kal17} yields the equation \eqref{4-eq:k toi}.
	
	To complete the proof, by Theorem 4.2 in \cite{bil68} (which can be viewed as a variant of the Slutsky's lemma), and 
	Theorem 4.11 in \cite{kal17}, it suffices to show that for every \(u>0\) and every function \(f\) as above
	\begin{align*}
		\lim_{K\to\infty}\limsup_{t\to\infty} \mathbb{P} \Bigg(\Bigg| \sum_{i\geq0} \sum_{j=1}^{L'_i} \mathbbm{1}_{\{|T'_i-t|\leq K\}} f({T'_i-t+T'_{ij}}) - \sum_{i\geq0} \sum_{j=1}^{L'_i} f({T'_i-t+T'_{ij}}) \Bigg| >u \Bigg)=0\,.
	\end{align*}
	Assume that the support of \(f\) is contained in \(\langle -C,C\rangle,\) for some \(C>0\). Observe that for $K >C$
	\begin{gather*}
		\mathbb{P}\Bigg(\Bigg| \sum_{i\geq0} \sum_{j=1}^{L'_i} \mathbbm{1}_{\{|T'_i-t|\leq K\}} f({T'_i-t+T'_{ij}}) 
		- \sum_{i\geq0} \sum_{j=1}^{L'_i} f({T'_i-t+T'_{ij}}) \Bigg|>u\Bigg)\\
		= 
		\mathbb{P}\Bigg(\Bigg| \sum_{i\geq0} \sum_{j=1}^{L'_i} \mathbbm{1}_{\{|T'_i-t|> K\}} f({T'_i-t+T'_{ij}})\Bigg| >u \Bigg)  \\
		\leq 
		\mathbb{P}\big(\exists i\geq0, j\leq L'_i: |T'_i-t|> K, |T'_i-t+T'_{ij}|< C\big)\,.
	\end{gather*}
	Denote 
	\[
	\sigma(t) = \inf\{i\geq0:T_i' > t \}
	\] and observe that the last expression above is bounded by 
	\begin{align*}  	
		\mathbb{P}\big(\exists i: &\phantom{.} T'_i-t <- K, T'_i-t+R'_i>-C\big)
		+
		\mathbb{P}\big(\exists i: T'_i-t> K, T'_i-t - R'_i< C\big)\\
		& \leq 
		\mathbb{P}\big( T'_0-t<- K, T'_0-t+R'_0>-C\big) \\
		& +  
		\sum_{i\geq 1}\mathbb{P} \big(T'_{i-1}-t<- K,T '_{i-1}-t+X_i'+R'_i>-C\big)\\
		& + 
		\mathbb{P} \big(T'_{\sigma(t+K)}-(t+K) - R'_{\sigma(t+K)} < C - K\big) \\
		& +
		\sum_{i\geq 1}\mathbb{P}\big(T'_{i-1} -t  > K, T'_{i-1}+X_i'-t-R'_i < C\big) \\
		& =:
		I_0^+(t, K) +  I_1^+(t, K) + I_0^-(t, K) + I_1^-(t, K)\,.
	\end{align*}
	It is easy to see that
	\begin{align*}
		\lim_{K\toi}\limsup_{t\toi} I_0^+(t, K)
		\leq
		\lim_{K\toi} \PP(R_0'>K-C)
		=
		0\,.
	\end{align*}
	Denote by \(U'=\sum_{i\geq0}F_{T'_i}\) the renewal function, where \(F_{T'_i}\) is the distribution function of the arrival time \(T'_i\). Then,
	using the independence between \(T'_{i-1}\) and $X'_i, R'_i$, for all \(i\geq1\), we obtain
	\begin{align*}
		I_1^+(t, K) = \sum_{i\geq1}\int_0^{t-K}\mathbb{P}(y-t+X_i'+R_i'>-C)\rmd F_{T'_{i-1}}(y)\\
		= 
		\int_0^{t-K}\mathbb{P}(X_1'+R'_1 + C - K> t -K -y)\rmd U'(y)\,.
	\end{align*}
	Observe that the function $y \mapsto \mathbb{P}(X_1'+R'_1 + C - K> y)$ is nonnegative, nonincreasing and Riemann integrable on \(\RR_+\), since
	\[
	\int_0 ^\infty \mathbb{P}(X_1'+R'_1 + C - K> y) \rmd y 
	= \int_{K-C}^\infty \mathbb{P}(X_1'+R'_1> y) \rmd y 
	\leq \EE[X_1'+R_1']
	<\infty\,,
	\] 
	by integrability assumptions on \(X'_1\) and \(R'_1\).
	By Remark 3.10.3 in \cite{res92}, it is also a directly Riemann integrable function on \(\RR_+\), hence by the key renewal theorem (see e.g.\ Theorem 3.10.1 in \cite{res92}) and integrability assumptions on \(X'_1\) and \(R'_1\)
	\[
	\lim_{K\toi}\limsup_{t\toi} I_1^+(t, K)
	=
	\lim_{K\toi} \frac{1}{\mu} \mathbb{E}(X_1'+R'_1 + C - K)_+
	=
	0\,.
	\]
	It is a simple consequence of Theorem \ref{lm0} and Theorem 4.11 in \cite{kal17} that
	\[
	(T'_{\sigma(t)}-t,R'_{\sigma(t)})\dto(T_0,R_0)\,,
	\]
	as \(t \toi\), where \(R_0 = {\sup_{j\leq L_0}\vert T_{0j}\vert}\). Hence, by the continuous mapping argument it follows that
	\[
	\lim_{K\toi}\limsup_{t\toi} I_0^-(t, K)
	=\lim_{K\toi}\mathbb{P}(T_{0} - R_0 < C - K)
	\leq \lim_{K\toi}\mathbb{P} (R_0> K-C)
	=0\,.
	\]
	As before, using independence assumption between \(T'_{i-1}\) and $X'_i, R'_i$, for all \(i\geq1\), we obtain
	\begin{gather*}
		I_1^-(t, K)
		= 
		\int_{t+K}^{\infty}\mathbb{P}(X_1'-R'_1 - C + K< t +K -y) \rmd U'(y)\\
		\leq
		\int_{t+K}^{\infty}\mathbb{P}(-R'_1 - C + K< t +K -y) \rmd U'(y)
		\,.
	\end{gather*}
    Note that the function \(g:\RR\to\RR_+\) given by $g(y)=\mathbb{P}(-R'_1 - C + K< y)$ is nonnegative, nondecreasing and, by integrability assumption on \(R'_1\), Riemann integrable on \(\langle-\infty,0]\).
	Another application of Remark 3.10.3 in \cite{res92} (now applied to \(g(-y)\)) yields that $g$  is a directly Riemann integrable function on \(\langle-\infty,0]\).
	Hence, by a variant of the key renewal theorem and  integrability assumptions on \(X'_1\) and \(R'_1\)
	\[
	\lim_{K\toi}\limsup_{t\toi} I_1^-(t, K)
	=
	\lim_{K\toi} \frac{1}{\mu} \mathbb{E}(-R'_1 - C + K)_-
	=
	0\,,
	\]
	which completes the proof.
	
\end{proof}

\subsection{Proof of Lemma~\ref{thm:ui}}

\begin{proof}
	Based on the position of the cluster centers \(T'_i\), we obtain the following decomposition
	\begin{align*}
		\xi'\langle t, t+x]
		=\xi'_{1}\langle t, t+x]+\xi'_{2}\langle t, t+x]+\xi'_{3}\langle t, t+x]\,,
	\end{align*}
	where
	\[
	\xi'_{1}\langle t, t+x]
	=\sum_{i\geq0}\sum_{j=1}^{L'_i} \mathbbm{1}_{\{T'_i+T'_{ij}\in\langle t, t+x]\}} \mathbbm{1}_{\{T'_i\leq t\}}\,,
	\]
	\[
	\xi'_{2}\langle t, t+x]
	=\sum_{i\geq0}\sum_{j=1}^{L'_i} \mathbbm{1}_{\{T'_i+T'_{ij}\in\langle t, t+x]\}} \mathbbm{1}_{\{ t< T'_i \leq t+x\}}\,
	\]
	and
	\[
	\xi'_{3}\langle t, t+x]
	=\sum_{i\geq0}\sum_{j=1}^{L'_i} \mathbbm{1}_{\{T'_i+T'_{ij}\in\langle t, t+x]\}} \mathbbm{1}_{\{T'_i> t+x\}}\,.
	\] 
	Note that if \(T'_{ij}\) take only negative values, \(\xi'_1\langle t, t+x]\) is identically 0. Hence, when dealing with this term, one needs to take care only of \(T'_{ij}\geq0\),  for all \(i,j\leq L'_i\). Similarly, when dealing with \(\xi'_3\langle t, t+x]\) one can ignore positive \(T'_{ij}\) and assume that \(T'_{ij}\leq0\), for all \(i,j\leq L'_i\). These simple observations will be useful in the sequel.
    Here we show uniform integrability of the term $\xi'_1$,
	the other two terms can be handled similarly, we refer to~\cite{MDthesis} for details. Once the uniform integrability of all three terms is proved, statement of the lemma follows by Theorem 4.6 in Chapter 5 of \cite{gut05}.

	Observe that 
	\begin{align*}
		\xi'_{1}\langle t, t+x]
		&\leq  \sum_{i\geq0}{L'_i} \mathbbm{1}_{\{T'_i+R'_i>t\}} \mathbbm{1}_{\{T'_i\leq t\}}\\
		&\leq  L'_0
		+  \sum_{i\geq1} L'_i \mathbbm{1}_{\{T'_{i}-t+R'_i>0\}} \mathbbm{1}_{\{T'_{i}\leq t\}}
		=:L'_0 + I'_1(t)\,.\numberthis\label{4-eq:e1}
	\end{align*}
	Trivially, an integrable random variable \(L'_0\) is uniformly integrable. For the other term in the above relation, using the independence between \(T'_{i-1}\) and \(X'_i\), \(L'_i\), \(R'_i\), for all \(i\geq1\), we get
	\begin{align*}
		\EE I'_1(t)
		&=\EE\Bigg[\sum_{i\geq1} \EE\Big[L'_i \mathbbm{1}_{\{T'_{i-1}+X'_i-t+R'_i>0\}} \mathbbm{1}_{\{T'_{i-1}+X'_i\leq t\}} \mathbbm{1}_{\{T'_{i-1}\leq t\}}|T'_{i-1} \Big] \Bigg]\\
		&=\sum_{i\geq1}\int_{0}^t\EE \Big[L'_{i} \mathbbm{1}_{\{X'_{i}+R'_{i}>t-y\}} \mathbbm{1}_{\{X'_i\leq t-y\}} \Big] \rmd F_{T'_{i-1}}(y)
		\\
		&=\int_0^t \EE\Big[ L'_1 \mathbbm{1}_{\{X'_1+R'_1>t-y\}} \mathbbm{1}_{\{X'_1\leq t-y\}} \Big] \rmd U'(y)\,.
	\end{align*}
	By the key renewal theorem, as \(t\toi\), the above expression converges to
	\[
	\frac{1}{\mu} \int_0^\infty \EE \Big[L'_1 \mathbbm{1}_{\{X'_1+R'_1>y\}} \mathbbm{1}_{\{X'_1\leq y\}} \Big] \rmd y
	=\frac{1}{\mu} \EE \Bigg[L'_1 \int_{X'_1}^{X'_1+R'_1} \rmd y\Bigg]
	=\frac{1}{\mu} \EE [L'_1R'_1]
	=:c_1\,,
	\]
	if 
	\(y\mapsto \EE\Big[ L'_1 \mathbbm{1}_{\{X'_1+R'_1>y\}} \mathbbm{1}_{\{X'_1\leq y\}} \Big]\) is a directly Riemann integrable function on \(\RR_+\). 
    For instance, this can be shown by Remarks 3.10.2, 3.10.5 in \cite{res92} and Assumption \ref{ass:fma}.

	Let \(R_i=\sup_{j\leq L_i}|T_{ij}|\), for all \(i\in\ZZ\). We now show that, as \(t\toi\), 
	\[
	I'_1(t)
	\dto
	\sum_{i\leq-1} {L_{i}} \mathbbm{1}_{\{T_{i}+R_{i}>0\}}=:I'_1\,,
	\]
	by checking the conditions of Theorem 4.2 from \cite{bil68}.
	Simple modifications in the proof of Theorem \ref{thm:ert} give
	\[
	\sum_{i\geq1} L'_i \delta_{T'_{i}-t+R'_i} \mathbbm{1}_{\{T'_{i}\leq t\}}
	\dto
	\sum_{i\leq -1}L_{i}\delta_{T_{i}+R_{i}}\,,
	\]
	hence Theorem 4.11 in \cite{kal17} yields
	\[
	\sum_{i\geq1} L'_i \mathbbm{1}_{\{0< T'_{i}-t+R'_i\leq M\}} \mathbbm{1}_{\{T'_{i}\leq t\}}
	\dto
	\sum_{i\leq-1}L_{i}\mathbbm{1}_{\{0< T_{i}+R_{i}\leq M\}}\,,
	\]
	as \(t\toi\), for an arbitrary \(M>0\).
	Considering the random variable on the right hand side, letting \(M\toi\) and applying the monotone convergence theorem, we get
	\[
	\sum_{i\leq-1}L_{i}\mathbbm{1}_{\{0< T_{i}+R_{i}\leq M\}}
	\asto 
	\sum_{i\leq-1}L_{i}\mathbbm{1}_{\{T_{i}+R_{i}>0\}}\,.
	\]
	It now remains to show that for every \(u>0\)
	\[
	\lim_{M\toi}\limsup_{t\toi}\PP\Bigg(\Bigg|\sum_{i\geq1} L'_{i} \mathbbm{1}_{\{0<T'_{i}-t+R'_{i}\leq M\}} \mathbbm{1}_{\{T'_{i}\leq t\}}
	-
	I'_1(t)
	\Bigg|>u\Bigg)=0\,.
	\]
	Indeed,
	\begin{align*}
		\PP \Bigg(\Bigg|\sum_{i\geq1} L'_{i} 
		&  
		\mathbbm{1}_{\{0<T'_{i}-t+R'_{i}\leq M\}} \mathbbm{1}_{\{T'_{i}\leq t\}}
		-\sum_{i\geq1} L'_{i} \mathbbm{1}_{\{ T'_{i}-t+R'_{i}>0\}}\mathbbm{1}_{\{T'_{i}\leq t\}}
		\Bigg|>u \Bigg)\\
		&\leq
		\PP(\exists i\geq1: T'_{i}-t+R'_{i}> M , T'_{i-1}\leq t)\\
		&\leq
		\sum_{i\geq1} \PP(T'_{i-1}+X'_{i}-t+R'_{i}> M, T'_{i-1}\leq t)\\
		&=
		\int_0^t \PP(X'_{1}+R'_{1}-M>t-y) \rmd U'(y)\,,
	\end{align*}
	which converges by the key renewal theorem to
	\[
	\frac{1}{\mu}\int_0^\infty \PP(X'_{1}+R'_{1}-M>y)\rmd y
	= \frac{1}{\mu}\EE(X'_1+R'_1-M)_+\,,
	\]
	as \(t\toi\). It is easy to show that the integrability assumptions on \(X'_1\) and \(R'_1\) assure that the conditions of the key renewal theorem are satisfied. Under the same integrability assumptions, we conclude that the above  expression converges to 0, as \(M \toi\).

	By construction, \(T_i\) is independent of \(R_i, L_i\), and \(T_{i}\stackrel{d}{=}-T_{-i-1}\), \(X_1\stackrel{d}{=}X'_1\), \(R_{i}\stackrel{d}{=}R'_{1}\) and \(L_{i}\stackrel{d}{=}L'_{1}\), for all \(i\leq-1\).
	Hence, we obtain
	\begin{align*}
		\EE I'_1&
		=\EE \Bigg[ \sum_{i\leq-1} L_{i}\mathbbm{1}_{\{T_{i}+R_{i}>0\}} \Bigg]
		=\EE\Bigg[\sum_{i=-\infty}^{-1} \EE \big[L_i\mathbbm{1}_{\{T_{i}+R_{i}>0\}}|T_i\big]\Bigg]\\
		&=\int_{-\infty}^0 \sum_{i=-\infty}^{-1} \EE\big[L_i\mathbbm{1}_{\{y+R_i>0\}}\big] \rmd F_{T_i}(y)
		=\int_{-\infty}^0 \sum_{i=0}^{\infty} \EE\big[L'_{1}\mathbbm{1}_{\{y+R'_{1}>0\}}\big] \rmd F_{T_i}(-y)\\
		&=\int_0^{\infty} \EE\big[L'_1\mathbbm{1}_{\{R'_1>y\}}\big] \rmd U(y)
		=\frac{1}{\EE X_1} \int_0^{\infty} \EE\big[L'_1\mathbbm{1}_{\{R'_1>y\}}\big] \rmd y
		=\frac{1}{\mu} \EE[L'_1R'_1]
		=c_1\,,
	\end{align*}
	which is finite by the assumptions.

	We have shown that 
	\[
	\EE I'_1(t)\to c_1\,,\quad
	I'_1(t)\dto I'_1
	\quad\text{and}\quad 
	\EE I'_1 = c_1<\infty\,,
	\]
	hence, by Lemma 5.11 in \cite{kal21}, \( I'_1(t)\), \(t\geq0\) are uniformly integrable random variables. Applying first Theorem 4.6 and then Theorem 4.5 in Chapter 5 of \cite{gut05} to \eqref{4-eq:e1}, we conclude that \(\xi'_1\langle t, t+x]\), \(t\geq0\) are also uniformly integrable random variables.

\end{proof}

\subsection{Proof of Lemma \ref{lem:UxiFin}}

\begin{proof}
	For all \(t\geq0\)
	\begin{align*}
		U_{\xi'}(t)
		&\leq \EE L'_0 
		+ \EE \Bigg[ \sum_{i=1}^\infty {L'_i} \mathbbm{1}_{\{T'_i - R'_i\leq t\}} \mathbbm{1}_{\{T'_i \leq t\}} \Bigg] 
		+ \EE \Bigg[ \sum_{i=1}^\infty {L'_i} \mathbbm{1}_{\{T'_i-R'_i\leq t\}} \mathbbm{1}_{\{T'_i> t\}} \Bigg]  \,. \numberthis\label{4-eq:u xi}
	\end{align*}
	The first term in \eqref{4-eq:u xi} is finite by the assumption. For the second term, observe that
    \begin{align*}
		\EE\Bigg[\sum_{i=1}^\infty L'_i \mathbbm{1}_{\{T'_i-R'_i\leq t\}} \mathbbm{1}_{\{T'_i\leq t\}} \Bigg]
		\leq \EE\Bigg[\sum_{i=1}^\infty L'_i \mathbbm{1}_{\{T'_{i-1}\leq t\}} \Bigg]
		= \EE L'_1 U'(t)
		<\infty\,,
	\end{align*}
	by the integrability assumption on \(L'_1\) and Theorem 3.3. in \cite{res92}.
	As in the proof of Lemma \ref{thm:ui}, one can show that as \(t\toi\)
	\begin{align}\label{4-eq:ufin}
		\EE \Bigg[ \sum_{i=1}^\infty {L'_i} \mathbbm{1}_{\{T'_i-R'_i\leq t\}} \mathbbm{1}_{\{T'_i> t\}} \Bigg]
		\to
		c<\infty\,,
	\end{align}
	where 
	\[
	c
	=\frac{1}{\mu} \EE [L'_1(X'_1-(X'_1-R'_1)_+)]
	+
	\frac{1}{\mu} \EE [L'_1(R'_1-X'_1)_+]\,.
	\] 
	Indeed, observe that
	\begin{align*}
		\EE &\Bigg[ \sum_{i=1}^\infty {L'_i} \mathbbm{1}_{\{T'_i-R'_i\leq t\}} \mathbbm{1}_{\{T'_i> t\}} \Bigg]\\
		&=\EE \Bigg[ \sum_{i\geq1} L'_i \mathbbm{1}_{\{T'_{i}-R'_i\leq t\}} \mathbbm{1}_{\{T'_{i}> t\}}\mathbbm{1}_{\{T'_{i-1}\leq t\}} \Bigg]
		+
		\EE \Bigg[ \sum_{i\geq1} L'_i \mathbbm{1}_{\{T'_{i}-R'_i\leq t\}} \mathbbm{1}_{\{T'_{i-1}> t\}} \Bigg]\\
		&=\int_0^{t} \EE\big[ L'_1 \mathbbm{1}_{\{X'_1-R'_1 \leq t-y\}} \mathbbm{1}_{\{X'_{1} > t-y\}} \big] \rmd U'(y)
		+
		\int_{t}^\infty \EE\big[ L'_1 \mathbbm{1}_{\{X'_1-R'_1 \leq t-y\}}\big] \rmd U'(y)\,,
	\end{align*}
	where we have used the independence between \(T'_{i-1}\) and \(X'_i\), \(R'_i\), for all \(i\geq1\), in the last equality. It is easy to show that integrability assumptions together with Remark 3.10.2 and Remark 3.10.3 in \cite{res92} justify the use of the key renewal theorem, hence
	\begin{align*}
		\EE \Bigg[ \sum_{i=1}^\infty {L'_i} \mathbbm{1}_{\{T'_i-R'_i\leq t\}} \mathbbm{1}_{\{T'_i> t\}} \Bigg]
		\to c\,.
	\end{align*}
By our integrability assumptions, the limiting value $c$ is finite. Therefore, the left hand is finite for all large enough $t$ as well. From \eqref{4-eq:u xi}, we can conclude that  the nondecreasing renewal function $U_{\xi'}(t)$ has to be finite for all $t$  sufficiently large. Finally, since it is nonnegative and nondecreasing,  the renewal function has to be finite for all \(t\).		

\end{proof}

\subsection{Proof of Corollary \ref{cor:KRT}}

Recall, the definition of the renewal function \(U_{\xi'}\) and the corresponding  renewal measure.
Let \(G\) denote distribution of \(X'_0\) and \(F\) common distribution of \(X'_i,i\geq1\). Consider first the classical renewal equation
\begin{align*}
U'(t)
=G(t)+F*U'(t),\quad t\geq0\,,
\end{align*}
where
$U'(t)
=\EE\Big[\sum_{i=0}^\infty \mathbbm{1}_{\{T'_i\leq t\}}\Big]
=\EE\Big[\sum_{i=0}^\infty \delta_{T'_i}[0,t]\Big]$.
Using the independence assumption between \(T'_{i-1}\) and \(L'_i, (T'_{ij})_j\), for all \(i\geq1\), we get that
\begin{align*}
U_{\xi'}(t)
&=\mathbb{E} \Bigg[ \sum_{i=0}^\infty\sum_{j=1}^{L'_i}\delta_{T'_i+T'_{ij}}\langle-\infty,t]\Bigg]\\
&= \phi(t) + \mathbb{E}\Bigg[\sum_{i=1}^{\infty} \mathbb{E} \Bigg[\sum_{j=1}^{L'_i} \mathbbm{1}_{\{T'_{i-1}+X'_i+T'_{ij}\leq t\}}|T'_{i-1}\Bigg]\Bigg]\\
&= \phi(t) + \int_0^\infty \sum_{i=1}^{\infty} \mathbb{E} \Bigg[\sum_{j=1}^{L'_i} \mathbbm{1}_{\{X'_i+T'_{ij}\leq t-y\}}\Bigg] \rmd F'_{T_{i-1}}(y)\\
&=\phi(t) + \int_0^\infty\psi(t-y) \rmd U'(y)\,,
\end{align*}
where \(\phi(t)=\mathbb{E}\Big[\sum_{j=1}^{L'_0}\mathbbm{1}_{\{T'_0+T'_{0j}\leq t\}}\Big]\) 
and 
\(\psi(t)=\mathbb{E} \Big[\sum_{j=1}^{L'_1} \mathbbm{1}_{\{X'_1+T'_{1j}\leq t\}}\Big]\), \(t\in\RR\). 
One now easily obtains a convolution equation
\begin{align*}
U_{\xi'}(t)
&=\phi(t) + U'*\psi(t)\\
&=\phi(t) + (G+F*U')*\psi(t)\\
&=(1-F)*\phi(t)+G*\psi(t)+F*\phi(t)+F*U'*\psi(t)\\
&=(1-F)*\phi(t)+G*\psi(t)+F*U_{\xi'}(t)\,,\numberthis\label{eq:conv}
\end{align*}
for all \(t\geq0\).

\begin{proof}[Proof of Corollary \ref{cor:KRT}]
	The proof is divided into three steps, with successively more complex \(g\).

	\noindent\textit{Step 1.} Suppose
	\[
	g(t)=\mathbbm{1}_{[(k-1)h,kh\rangle}(t)\,,\quad t\geq0\,,
	\]
	for a fixed \(k\in\mathbb{N}\) and \(h>0\). Then \(g(t-y)=1\) if and only if \(t-kh<y\leq t-(k-1)h\), thus
	\begin{align*}
	\int_0^t g(t-y) \rmd U_{\xi'}(y)
	=U_{\xi'}\langle t-kh, t-(k-1)h]\mathbbm{1}_{\{t\geq kh\}}
	+U_{\xi'}\langle 0, t-(k-1)h]\mathbbm{1}_{\{(k-1)h\leq t< kh\}}\,.
	\end{align*}
	Observe that for all \(t\in [ (k-1)h,kh\rangle\),
	\begin{equation}\label{4-eq:2nd part}
	U_{\xi'}\langle 0, t-(k-1)h]
	\leq U_{\xi'}(h)
	<\infty\,,
	\end{equation}
	hence Corollary \ref{cor:brt} yields
	\[
	\lim_{t\to\infty} \int_0^t g(t-y)\rmd U_{\xi'}(y)
	= \frac{1}{\mu}\EE L'_1h+0
	= \frac{1}{\mu}\EE L'_1 \int_0^\infty g(y)\rmd y\,.
	\]

	\noindent\textit{Step 2.} Suppose
	\[
	g(t)=\sum_{k\geq1}c_k\mathbbm{1}_{[(k-1)h,kh\rangle}(t)\,,\quad t\geq0\,,
	\]
	where \((c_k)_{k\geq1}\) is a sequence of nonnegative numbers such that \(\sum_{k\geq1}c_k<\infty\) and \(h\) is chosen so that \(F(h)<1\). Then
	\begin{align*}
	\int_0^t g(t-y)\rmd U_{\xi'}(y)
	&= \sum_{k=1}^\infty c_k
	U_{\xi'}\langle t-kh, t-(k-1)h]\mathbbm{1}_{\{t\geq kh\}}\\
	&+ \sum_{k=1}^\infty c_k U_{\xi'}\langle 0, t-(k-1)h]\mathbbm{1}_{\{(k-1)h \leq t < kh\}}\,.
	\end{align*}
	For each \(k\in \NN\) and \(t\geq kh\), we have
	\[
	\lim_{t\to\infty} 
	U_{\xi'}\langle t-kh, t-(k-1)h]
	= \frac{1}{\mu}\EE L'_1 h\,.
	\]
	It follows from the equation (\ref{eq:conv}), that 
	\[
	\phi(t)+G*\psi(t)
	\geq(1-F)*\phi(t)+G*\psi(t)
	=(1-F)*U_{\xi'}(t)\,,
	\]
	hence
	\begin{align*}
	\EE L'_0 + \EE L'_1
	& \geq \EE L'_0
	+ \int_0^\infty \EE L'_1 \rmd G(y)\\
	& \geq \phi(t-(k-1)h)
	+ \int_0^\infty\psi(t-(k-1)h-y) \rmd G(y)\\
	& \geq\int_{t-kh}^{t-(k-1)h}(1-F(t-(k-1)h-y)) \rmd U_{\xi'}(y)\\
	& \geq (1-F(h)) U_{\xi'}\langle t-kh, t-(k-1)h]\,.
	\end{align*}
	Therefore
	\[
	\sup_{t,k}
	U_{\xi'}\langle t-kh, t-(k-1)h]
	\leq \frac{\EE L'_0 + \EE L'_1}{1-F(h)}\,,
	\]
	which is finite by integrability assumptions on \(L'_0\), \(L'_1\) and choice of \(h\). For an arbitrary \(k\) and \(t\in[(k-1)h,kh\rangle\), recall \eqref{4-eq:2nd part}. Hence, by the dominated convergence theorem 
	\[
	\lim_{t\to\infty} \int_0^t g(t-y)\rmd U_{\xi'}(y)
	=\frac{1}{\mu}\EE L'_1 h\sum_{k\geq1}c_k + 0
	=\frac{1}{\mu}\EE L'_1\int_0^\infty g(y) \rmd y\,.
	\]

	\noindent\textit{Step 3.} Let \(g\) be an arbitrary directly Riemann integrable function on \(\RR_+\) and define
	\begin{align*}
	\overline{g}(t)&=\sum_{k=1}^\infty \sup_{(k-1)h\leq y<kh}g(y) \mathbbm{1}_{[(k-1)h,kh\rangle}(t)\\
	\underline{g}(t)&=\sum_{k=1}^\infty \inf_{(k-1)h\leq y<kh}g(y) \mathbbm{1}_{[(k-1)h,kh\rangle}(t)\,,
	\end{align*}
	for all \(t\geq0\).
	By the definition of direct Riemann integrability on \(\RR_+\) (see e.g.\ Chapter XI in \cite{fel71})
	\[
	\sum_{k=1}^\infty \inf_{(k-1)h\leq y<kh} g(y)
	\leq \sum_{k=1}^\infty \sup_{(k-1)h\leq y<kh} g(y)
	<\infty\,,
	\]
	thus \(\overline{g}\) and \(\underline{g}\) have the same structure as the functions considered in Step 2. Denote
	\begin{align*}
	\overline{\sigma}(h)=h \sum_{k=1}^\infty \sup_{(k-1)h\leq y<kh}g(y) \quad \text{ and } \quad
	\underline{\sigma}(h)=h \sum_{k=1}^\infty \inf_{(k-1)h\leq y<kh}g(y) \,.
	\end{align*}
	It follows that
	\begin{align*}
	\frac{1}{\mu}\EE L'_1\underline{\sigma}(h)
	&= \liminf_{t\to\infty} \int_0^t \underline{g}(t-y)\rmd U_{\xi'}(y)
	\leq\liminf_{t\to\infty} \int_0^t g(t-y)\rmd U_{\xi'}(y)\\
	&\leq\limsup_{t\to\infty} \int_0^t g(t-y)\rmd U_{\xi'}(y)
	\leq\limsup_{t\to\infty} \int_0^t \overline{g}(t-y)\rmd U_{\xi'}(y)
	=\frac{1}{\mu}\EE L'_1\overline{\sigma}(h)\,,\numberthis\label{4-eq:krt ineq}
	\end{align*}
	since \(\underline{g}\leq g\leq \overline{g}\). By the definition of direct Riemann integrability on \(\RR_+\)
	\[
	\lim_{h\to0}(\overline{\sigma}(h)-\underline{\sigma}(h))=0
	\] and by Remark 3.10.2 in \cite{res92}
	\[
	\lim_{h\to0}\overline{\sigma}(h)=\int_{0}^\infty g(y)\rmd y\,.
	\]
	Hence, letting \(h\to0\) in \eqref{4-eq:krt ineq} the result follows.
\end{proof}

\section*{Acknowledgment}
We sincerely thank the anonymous referee for careful reading of the manuscript and several very helpful comments. The work of Bojan Basrak was financed in part by the Croatian-Swiss Research Program of the Croatian Science Foundation and the Swiss National Science Foundation - grant IZHRZ0 - 180549 and HRZZ grant IP-2022-10-2277.

\end{document}